\documentclass{article}

\usepackage{amsmath, amssymb, amsthm, bm}
\usepackage{graphicx, color}
\usepackage{hyperref}

\usepackage{geometry}
\newcommand{\normmm}[1]{{\left\vert\kern-0.25ex\left\vert\kern-0.25ex\left\vert #1
    \right\vert\kern-0.25ex\right\vert\kern-0.25ex\right\vert}} 

\newtheorem{Proposition}{Proposition}
\newtheorem{Lemma}[Proposition]{Lemma}

\newtheorem{Definition}[Proposition]{Definition}
\newtheorem{Theorem}[Proposition]{Theorem}

\begin{document}

\title{Analysis of the time domain acoustic scattering from open cavities}

\author{Bo Chen \footnote{College of Science, Civil Aviation University of China, Tianjin, P. R. China; {\em Email: charliecb@163.com}}\and
Fuming Ma  \footnote{School of Mathematics, Jilin University, Changchun, P. R. China.  {\em Email: mfm@jlu.edu.cn}}\and
Yukun Guo \footnote{Department of Mathematics, Harbin Institute of Technology, Harbin, P. R. China. {\em Email: ykguo@hit.edu.cn}}
}
\date{}

\maketitle

\begin{abstract}
This paper is concerned with the analysis of the time domain acoustic scattering from two-dimensional open cavities. A transparent boundary condition is developed to reformulate the scattering problem into an equivalent initial boundary value problem in the interior domain of the open cavity. The well-posedness, that is, the existence, uniqueness and stability of the solution to the reduced problem are studied via a ``Laplace domain'' to time domain analysis. Moreover, time domain boundary integral equations for the reduced problem are established.

{\bf Keywords}: time domain; acoustic scattering; open cavity; transparent boundary condition; well-posedness; boundary integral equation
\end{abstract}


\section{Introduction}

The scattering problems of acoustic waves have attracted extensive attention due to their significance in industry and medical equipment. Roughly speaking, the acoustic scattering problems can be divided into the frequency domain problem with time-harmonic or nearly time-harmonic wave field and the time domain problem with time-dependent non-harmonic wave field. Frequency domain problem, which mainly deals with the Helmholtz equation, can be taken as a simplification of the time domain case and is a more mature research \cite{Kress1983Integral, Zhang2013A, wang2017fourier, Liu2008On}. The analytical methods of time domain problems are diversiform, such as the direct analysis \cite{wang2017mathematical} and the analysis related to the so-called ``Laplace domain'' problem given by the Fourier-Laplace transform \cite{sayas2011retarded, ChenB2016time}. Since time domain scattering problems arise more naturally in application areas, in recent years, time domain scattering \cite{ha2003retarded, Laliena2009Theoretical, Gao2016Analysis} and inverse scattering \cite{Guo2015The, Monk2016An, Yukun2016A, guo2013toward} problems have attracted more and more attentions.

We mainly care about the scattering of acoustic waves in homogeneous, isotropic background medium. In this paper, the unbounded scatterer is chosen as two-dimensional open cavities embedded into the half-plane. The corresponding frequency domain problem has been studied in \cite{Ammari2000An, Feng2005Uniqueness, Bao2005A,Lai2014A,Li2012An}. In recent years, the time domain electromagnetic scattering from open cavities has been mathematically studied in \cite{Li2015Analysis, Van2006Analysis,Zhiming2008ON,Li2018A}. However, to the best of our knowledge, there was no rigorous mathematical analysis of the time domain acoustic scattering from open cavities.

In comparison to the scattering from a bounded scatterer, the unbounded nature of the open cavities makes the analysis more challenging. In order to overcome this difficulty, we develop a transparent boundary condition (TBC) to reformulate the original scattering problem with unbounded scatterers equivalently into an initial boundary value problem in the interior domain of the open cavity. Then the well-posedness of the solution to the reduced problem is given via a Fourier argument and the analysis of the so-called ``Laplace domain'' problem given by the Fourier-Laplace transform of the time domain problem. Finally, a retarded potential boundary integral equation (RPBIE) method \cite{ha2003retarded} is used to solve the reduced problem and the convolution quadrature (CQ) method is involved to turn the calculation of the time domain problem into the calculation of the classic frequency problems. The corresponding integral equation method in frequency domain is analyzed in \cite{Jiang2012Numerical, Ammari2002Analysis, Li2013A, Sun2017Indirect, Bao2012Stability}. Our work on the time domain scattering from open cavities is inspired by these frequency domain investigations and the related time domain analyses of the Maxwell equations and the acoustic waves \cite{sayas2011retarded, Banjai2012Wave, Banjai2013Stable}.

The outline of this paper is as follows. We present the model scattering problem and the relevant spaces in Section 2. Then a TBC is developed to reformulate the time domain scattering problem into a reduced initial boundary value problem in a bounded domain in Section 3. In Section 4, the well-posedness of the reduced problem and the equivalence of the two time domain scattering problems are proved. In Section 5, RPBIEs for the reduced time domain problem and the CQ method for the computation of the time domain problems are established. The last section is devoted to the conclusion of this paper, as well as comments on our future works.

\section{Problem setting}\label{sec301}
\subsection{Model problem}
Consider the scattering of transient acoustic waves by an open cavity embedded in a ground plane. The space above the ground plane is filled with homogeneous background medium. The ground plane and the wall of the cavity are assumed to be sound-soft. Adopting Cartesian coordinates $(x_1,x_2,x_3)$, the cavity and the incident field are both assumed to be invariant with respect to $x_3$. Thus the three dimensional scattering problem can be simplified to the two dimensional case.

The incident field is chosen as the cylindrical wave emitted by a line source parallel to the $x_3$-axis in the half-space $\{(x_1,x_2,x_3)\in\mathbb{R}^3:x_2>0\}$. Denote by $y_{3d}:=(y,y_3)$ the coordinate of a source point with  $y:=(y_1,y_2)\in\mathbb{R}^2$. Consider that a causal signal $\lambda(t)$ (in particular $\lambda(t)=0$ for $t<0$) is simultaneously emitted by all the source points on the excitation line. Then the incident field has the form (see \cite{sayas2011retarded})
\begin{equation*}
u^i(t,x;y):=k(t,x;y)\ast \lambda(t),\quad  t\in \mathbb{R}, \ x\in \mathbb{R}^2\backslash\{y\},
\end{equation*}
where $k\ast\lambda$ is the time convolution of $k$ and $\lambda$ and
$$
k(t,x;y):=\frac{H(t-c^{-1}|x-y|)}{2\pi\sqrt{t^2-c^{-2}|x-y|^2}}
$$
is the Green's function of the operator $c^{-2}\partial_{tt}-\Delta$ in the free space $\mathbb{R}\times\mathbb{R}^2$. In this paper, $\Delta$ is the Laplacian in $\mathbb{R}^2$, $\partial_{t}=\partial/\partial t$, $\partial_{tt}=\partial^2/\partial t^2$, $H$ is the Heaviside step function and $c$ is the constant wave speed of the homogeneous background medium. For the sake of simplicity, we choose $c\equiv 1$ throughout the rest of this paper.

For the two-dimensional scattering problem, denote by $\mathbb{R}_+^2:=\{(x_1,x_2)\in\mathbb{R}^2:x_2>0\}$ the upper half-plane and $\mathbb{R}_0^2:=\{(x_1,x_2)\in\mathbb{R}^2:x_2=0\}$ the $x_1$-axis. The source point $y$ is assumed to be located in the upper half-plane $\mathbb{R}_+^2$. For $x=(x_1,x_2)$, define $x^\rho:=(x_1,-x_2)$. Then the reflected field $u^\rho$ is defined as
\begin{equation*}
u^\rho(t,x;y):=-k(t,x;y^\rho)\ast \lambda(t), \quad t\in\mathbb{R},\, x\in \mathbb{R}^2\backslash\{y^\rho\}.
\end{equation*}

\begin{figure}
\centering
\includegraphics[width=0.6\textwidth]{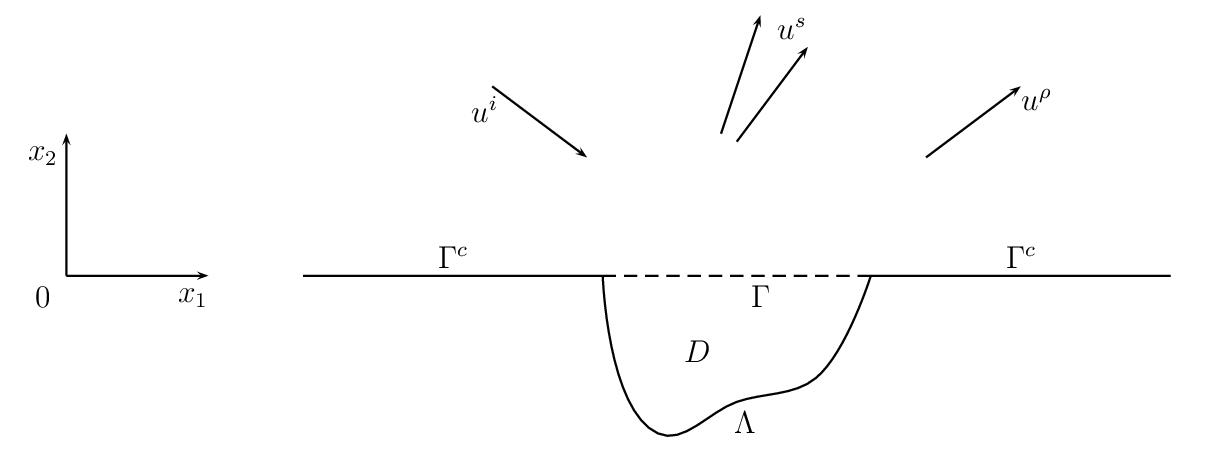}
\caption{Sketch of the two dimensional scattering problem.}\label{fig301}
\end{figure}

We refer to Figure \ref{fig301} for a geometrical illustration of the model scattering problem. Denote by $D$ the cavity with the boundary $\partial D=\Gamma\cup \Lambda$. The ground line is $\mathbb{R}_0^2=\Gamma\cup\Gamma^c$. The wall $\Lambda$ of the cavity is assumed to be $C^2$-smooth and $\mathbb{R}_0^2$ is not a tangent to $\Lambda$.

The total field $u$, which is divided into the incident field $u^i$, the reflected field $u^{\rho}$ and the scattered field $u^s$, satisfies the wave equation, Dirichlet boundary condition and initial conditions
\begin{align}
\label{301} u_{t t}-\Delta u=f & \quad \text{in}\;\;\mathbb{R}\times (D\cup\mathbb{R}_+^2), \\
\label{302}u=0  &\quad \text{on}\;\;\mathbb{R}\times(\Gamma^c\cup \Lambda),\\
\label{303}u(0,\cdot)=u_t(0,\cdot)=0  &\quad \text{in}\;\;D\cup\mathbb{R}_+^2,
\end{align}
In which the source term $f$ satisfies $\mathrm{supp}(f)\subset\mathbb{R}_+\times\mathbb{R}_+^2$.


\subsection{Space-time Sobolev spaces}
We recall some notation concerning Sobolev spaces (for details, see \cite{sayas2011retarded, Chen2010sampling}). Given a Lipschitz domain $\Omega\in\mathbb{R}^d$, denote
$$
(u,v)_{\Omega}:=\int_{\Omega}u v,\quad (\nabla u,\nabla v)_{\Omega}:=\int_{\Omega}\nabla u\cdot \nabla v.
$$
On this basis, define the norm with
$$\|u\|_{\Omega}:=\left[(u,\overline{u})_{\Omega}\right]^{1/2},$$
where the overline denotes the complex conjugate. Then the $H^1(\Omega)$-norm is
$$\|u\|_{H^1(\Omega)}:=\left(\|u\|_{\Omega}^2+\|\nabla u\|_{\Omega}^2\right)^{1/2}.$$
For $c>0$, define $\underline{c}:=\min\{1,c\}$ and the norm
$$\normmm{u}_{c,\Omega}:=\left(\|\nabla u\|_{\Omega}^2+c^2\|u\|_{\Omega}^2\right)^{1/2}.$$
Define the space
$$H_\Delta^1(\Omega):=\{u\in H^1(\Omega):\Delta u\in L^2(\Omega)\}$$
with the norm
$$\|u\|_{\Delta}:=\left(\|u\|_{\Omega}^2+\|\nabla u\|_{\Omega}^2+\|\Delta u\| _{\Omega}^2\right)^{1/2}.$$

Then we consider the trace spaces on $\Gamma$. Denote by
$$\langle \xi,\eta \rangle _{\Gamma}:=\int_{\Gamma}\xi\eta\,\mathrm{d}_{\Gamma}$$
the $L^2(\Gamma)$ inner product on $\Gamma$. Define
\begin{equation*}
\tilde{H}^{1/2}(\Gamma):=\{V\in H^{1/2}(\Gamma):\,E_0(V)\in H^{1/2}(\mathbb{R}^2_0)\}.
\end{equation*}
where $E_0$ is the extension operator from $H^{1/2}(\Gamma)$ to $H^{1/2}(\mathbb{R}^2_0)$ defined by
\begin{equation*}
E_0(V)(x):=
\begin{cases}
V(x), & x\in\Gamma, \\
0, & x\in\mathbb{R}^2_0\backslash\Gamma.
\end{cases}
\end{equation*}
In fact, $\tilde{H}^{1/2}(\Gamma)$ is the dual of $H^{-1/2}(\Gamma)$ with respect to the $L^2(\Gamma)$-inner product \cite{Ammari2000An,Hsiao2008Boundary}.
In consideration of the scattering problem, we define the spaces
\begin{align*}
H_E^1(D):=&\{V\in H^1(D):\,V|_\Lambda=0\;\text{and}\;\;V|_{\Gamma}\in\tilde{H}^{1/2}(\Gamma)\},\\
H_{\Delta,E}^1(D):=&\{V\in H_{\Delta}^1(D):\,V|_\Lambda=0\;\text{and}\;\;V|_{\Gamma}\in\tilde{H}^{1/2}(\Gamma)\},\\
H_{\Delta,E}^1(D\cup\mathbb{R}^2_+):=&\{V\in H_{\Delta}^1(D\cup\mathbb{R}^2_+):\,V=0\;\text{on}\;\;\Gamma^c\cup \Lambda\}
\end{align*}
with the norm of $H^1(D)$, $H_{\Delta}^1(D)$ and $H^1(D\cup\mathbb{R}^2_+)$, respectively.

Denote $\mathbb{C}_\sigma:=\{\omega\in\mathbb{C}:\,\mathrm{Im}(\omega)\geq\sigma>0\}$ and in particular $\mathbb{C}_+:=\{\omega\in\mathbb{C}:\,\mathrm{Im}(\omega)>0\}$. The Fourier-Laplace transform is defined by
\begin{equation}\label{330}
\mathcal{L}[f](\omega):=\int_{-\infty}^{\infty}\mathrm{e}^{\mathrm{i}\omega t}f(t)\,\mathrm{d}t,\quad \omega \in\mathbb{C}_\sigma.
\end{equation}
Correspondingly, the inversion formula is
\begin{equation}\label{331}
\mathcal{L}^{-1}[\varphi](t):=\frac{1}{2\pi}\int_{-\infty+\mathrm{i}\sigma}^{\infty+\mathrm{i}\sigma}\mathrm{e}^{-\mathrm{i}\omega t}\varphi(\omega)\,\mathrm{d}\omega.
\end{equation}

To analyse the time domain scattering problem, we recall some notation concerning space-time Sobolev spaces. For a Hilbert space $X$, denote by $\mathcal{D'}(X)$ and $\mathcal{S'}(X)$ the space of $X$-valued distributions and tempered distributions on the real line, respectively. For $\sigma\in\mathbb{R}$, define the spaces
$$\mathcal{L}'_\sigma(\mathbb{R},X):=\{f\in \mathcal{D'}(X):\,\mathrm{e}^{-\sigma t}f\in \mathcal{S'}(X)\}$$
and
$$\mathcal{L}'_\sigma(\mathbb{R}_{+},X):=\left\{f\in \mathcal{L}'_\sigma(\mathbb{R},X):\,f(t)=0,\;\forall t<0\right\}.$$
For $\sigma\in\mathbb{R}$ and $p\in\mathbb{R}$, define the space
$$
H_\sigma^p(\mathbb{R},X):=\left\{f\in\mathcal{L}'_\sigma(\mathbb{R},X):\,\int_{-\infty+\mathrm{i}\sigma}^{\infty+\mathrm{i}\sigma}|\omega |^{2p}\|\mathcal{L}[f](\omega )\|_X^2\,\mathrm{d}\omega <\infty\right\}
$$
with the norm
\begin{equation}\label{001}
\|f\| _{H_\sigma^p(\mathbb{R},X)}:=\left(\int_{-\infty+\mathrm{i}\sigma}^{\infty+\mathrm{i}\sigma}|\omega |^{2p}\|\mathcal{L}[f](\omega )\|_X^2\,\mathrm{d}\omega\right)^{1/2}.
\end{equation}

Taking into account the causality, define the space
$$
H_\sigma^p(\mathbb{R}_{+},X):=\left\{f\in H_\sigma^p(\mathbb{R},X):\,f(t)=0,\;\forall t<0\right\}
$$
with the norm of $H_\sigma^p(\mathbb{R},X)$.


\section{The reduced problem}\label{sec302}

In this section, a TBC is proposed to formulate an equivalent initial boundary problem of \eqref{301}--\eqref{303} in the bounded domain $D$. Since the source points are in the upper half-plane $\mathbb{R}_+^2$ and separated from the unbounded scatterer, the scattered field $u^s:=u-u^i-u^\rho$ satisfies the homogeneous wave equation (\cite{sayas2011retarded})
\begin{equation}\label{304}
u^s_{t t}-\Delta u^s=0  \quad \text{in}\;\;\mathbb{R}\times \mathbb{R}_+^2.
\end{equation}
Define $u^0:=u|_{\mathbb{R}\times\mathbb{R}_0^2}$. Note that $u^i+u^\rho=0$ on $\mathbb{R}\times\mathbb{R}_0^2$. Then
\begin{equation}\label{305}
u^s=u^0 \quad \text{on} \;\mathbb{R}\times\mathbb{R}_0^2.
\end{equation}

We intend to establish a TBC on $\Gamma$. For the corresponding frequency domain problem \cite{Ammari2000An}, the Fourier transform with respect to $x_1$ is employed to get the differential equations of $x_2$. However, since an additional variable $t$ is involved for time domain problems, an additional integral transform is needed for the analysis. After some serious thought, we find the Fourier-Laplace transform to be effective and befitting the well-posedness analysis.


\subsection{TBC in the ``Laplace domain''}

Taking formally the Fourier-Laplace transform of \eqref{304} and \eqref{305} with respect to $t$ implies
\begin{align}
\label{306} \Delta U^s(\omega,\cdot)+\omega^2U^s(\omega,\cdot)=0 & \quad \text{in}\;\;\mathbb{R}_+^2,\\
\label{307} U^s(\omega,\cdot)=U^0(\omega,\cdot) & \quad \text{on}\ \mathbb{R}_0^2,
\end{align}
where $\omega\in\mathbb{C}_+$, $U^s$ and $U^0$ are respectively the Fourier-Laplace transform of $u^s$ and $u^0$ with respect to $t$.

Furthermore, taking the Fourier transform of \eqref{306} and \eqref{307} with respect to $x_1$ yields
\begin{align}
\label{309}\left(\partial_{x_2}^2+(\omega^2-\xi_1^2)\right)\mathcal{F}_{x_1}[U^s](\omega,\xi_1,x_2)=0,&\quad x_2>0,\\
\label{310}\mathcal{F}_{x_1}[U^s](\omega,\xi_1,x_2)=\mathcal{F}_{x_1}[U^0](\omega,\xi_1,x_2),&\quad x_2=0.
\end{align}

The causality of the time domain problem implies the finite energy of the acoustic wave at each time instant \cite{sayas2011retarded}. Then $\omega\in\mathbb{C}_+$ and the finite energy implies
$$\mathcal{F}_{x_1}[U^s]=\mathrm{e}^{\mathrm{i} \omega\sqrt{1-\xi_1^2/\omega^2}x_2} \mathcal{F}_{x_1}[U^0].$$
Here $\sqrt{a}$ is the principle square root of $a\in\mathbb{C}$, that is, $\mathrm{Re}(\sqrt{a})\geq 0$. For $\omega\in\mathbb{C}_+$, set $\omega=\eta+\mathrm{i}\sigma$, $\sigma>0$. Notice that
$$
\eta\mathrm{Im}\left(1-\frac{\xi_1^2}{\omega^2}\right)=\frac{2\xi_1^2\eta^2\sigma}{(\eta^2+\sigma^2)^2}\geq 0.
$$
Moreover, the definition of the principle square root implies
$$
\mathrm{Im}\left(1-\frac{\xi_1^2}{\omega^2}\right)\mathrm{Im}\left(\sqrt{1-\frac{\xi_1^2}{\omega^2}}\right)\geq 0.
$$
Thus
$$
\mathrm{Re}\left(\mathrm{i}\omega\sqrt{1-\frac{\xi_1^2}{\omega^2}}\right)= -\sigma\mathrm{Re}\left(\sqrt{1-\frac{\xi_1^2}{\omega^2}}\right)-\eta\mathrm{Im}\left(\sqrt{1-\frac{\xi_1^2}{\omega^2}}\right)\leq 0.
$$

The inverse Fourier transform implies
$$
U^s(\omega,x_1,x_2)=\frac{1}{2\pi}\int_{\mathbb{R}}\mathrm{e}^{\mathrm{i}\xi_1x_1} \mathrm{e}^{\mathrm{i} \omega\sqrt{1-\xi_1^2/\omega^2}x_2} \mathcal{F}_{x_1}[U^0](\omega,\xi_1,0)\,\mathrm{d}\xi_1.
$$
Then we have
\begin{equation}\label{311}
\frac{\partial U^s}{\partial \nu}\bigg|_{x_2=0}=\mathcal{T}^F U^0,
\end{equation}
where the operator $\mathcal{T}^F$ is defined by
\begin{equation*}
(\mathcal{T}^F V)(\omega,x_1,x_2):=\frac{\mathrm{i} \omega}{2\pi}\int_{\mathbb{R}}\mathrm{e}^{\mathrm{i}\xi_1x_1} \sqrt{1-\xi_1^2/\omega^2} \mathcal{F}_{x_1}[E_0(V)](\omega,\xi_1,0)\, \mathrm{d}\xi_1.
\end{equation*}

Define
\begin{equation*}
G(\omega,\cdot):=\frac{\partial (U^i(\omega,\cdot)+U^\rho(\omega,\cdot))}{\partial \nu}\quad \text{on} \;\Gamma,
\end{equation*}
where $U^i$ and $U^\rho$ are, respectively, the Fourier-Laplace transform of $u^i$ and $u^\rho$ with respect to $t$. Then we have the boundary condition
$$
\frac{\partial U(\omega,\cdot)}{\partial \nu}=\mathcal{T}^F(\omega,\cdot)U(\omega,\cdot)+G(\omega,\cdot) \quad \text{on} \;\Gamma,
$$
which is a TBC in the so-called ``Laplace domain''.


\subsection{TBC in the time domain}

Then the inverse Fourier-Laplace transform is needed to formulate a TBC back in the time domain. Note that there are restrictions to use the strong inversion formula \eqref{331}. Consider the inverse Fourier-Laplace transform of $F(\omega)$, $\mathrm{Im}(\omega)=\sigma>0$. Assume that $F(\omega)$ satisfies
\begin{equation}\label{assume}
  |F(\omega)|\leq C_F(\sigma)|\omega|^{\mu},
\end{equation}
in which $\mu\in\mathbb{R}$ and $C_F:\mathbb{R}_+\rightarrow\mathbb{R}_+$ is a non-increasing function such that
$$C_F(\sigma)\leq\frac{M}{\sigma^\iota},\quad \forall \sigma\in (0,1),$$
where the constants $\iota>0,M>0$.

Again, set $\omega=\eta+\mathrm{i}\sigma$. When $\mu<-1$, the inverse Fourier-Laplace transform is defined as
$$f(t)=\mathcal{L}^{-1}[F](t):=\frac{1}{2\pi}\int_{-\infty+\mathrm{i}\sigma}^{\infty+\mathrm{i}\sigma}\mathrm{e}^{-\mathrm{i}\omega t}F(\omega)\mathrm{d}\omega=\frac{1}{2\pi}\int_{-\infty}^{\infty}\mathrm{e}^{\sigma t}\mathrm{e}^{-\mathrm{i}\eta t}F(\eta+\mathrm{i}\sigma)\mathrm{d}\eta.$$
Denote
$$\chi=\frac{\sigma^2}{\sigma^2+\eta^2}.$$
The assumption \eqref{assume} implies that $f(t)$ is well defined for $t\in\mathbb{R}$ and
$$|f(t)|\leq\frac{1}{2\pi}C_F(\sigma)\sigma^{1+\mu}\mathrm{e}^{\sigma t}B\left(\frac{1}{2},-\frac{\mu+1}{2}\right),$$
where $B$ is the Euler beta function
$$B(x,y)=\int_0^1\chi^{x-1}(1-\chi)^{y-1}\,\mathrm{d}\chi.$$

A contour integration argument implies that $f(t)$ is independent of $\sigma>0$. Taking the limit as $\sigma\rightarrow\infty$, we see that $f(t)=0$, $\forall t<0$. Moreover, set $\sigma=t^{-1}$ for $t>1$, we have the estimation
$$|f(t)|\leq M t^{\iota-(\mu+1)}B\left(\frac{1}{2},-\frac{\mu+1}{2}\right).$$
Then $f(t)$ is a causal function with polynomial growth. Obviously we have $f(t)\in\mathcal{L}'_\sigma(\mathbb{R}_{+},X)$.

For $\mu\geq -1$, we write $F(\omega)=(-\mathrm{i}\omega)^m F_m(\omega)$ with the integer $m>\mu+1$. Then we obtain
$$|F_m(\omega)|\leq C_F(\sigma)|\omega|^{\mu-m}$$
with $\mu-m<-1$. Then there exists a causal function $f_m(t)$ with polynomial growth such that $f_m(t)=\mathcal{L}^{-1}[F_m](t)$. On account of $\mathcal{L}[f_m^{(m)}](\omega)=(-\mathrm{i}\omega)^m \mathcal{L}[f_m](\omega)$, the inverse Fourier-Laplace transform of $F(\omega)$ is $f(t)=f_m^{(m)}(t)$. Then we also have $f(t)\in\mathcal{L}'_\sigma(\mathbb{R}_{+},X)$ for this case. For more details, see the analysis of Sayas \cite{sayas2011retarded} and Lubich \cite{lubich1994multistep}.

From the above analysis, there is always an inverse Fourier-Laplace transform of $F(\omega)$ with or without the strong inversion formula \eqref{331}. For simplification, assume that the strong inversion formula can be used in this paper. Back to time domain, we can get the following definition of the boundary operator and the TBC.
\begin{Definition}
The boundary operator $\mathcal{T}: \mathcal{L}'_\sigma(\mathbb{R},\tilde{H}^{1/2}(\Gamma))\rightarrow \mathcal{L}'_\sigma(\mathbb{R},H^{-1/2}(\Gamma))$ is defined as
\begin{equation*}
\mathcal{T}v:=\frac{1}{4\pi^2}\int_{-\infty+\mathrm{i}\sigma}^{\infty+\mathrm{i}\sigma}\int_{-\infty}^\infty \mathrm{e}^{-\mathrm{i}\omega t+i\xi_1x_1} \mathrm{i}\omega\sqrt{1-\xi_1^2/\omega^2}\mathcal{L}\circ\mathcal{F}_{x_1}[E_0(v)](\omega,\xi_1,0) \,\mathrm{d}\xi_1\mathrm{d}\omega.
\end{equation*}
\end{Definition}

\begin{Definition}
The transparent boundary condition (TBC) of the time domain scattering problem \eqref{301}--\eqref{303} is defined as
\begin{equation*}
\frac{\partial u}{\partial \nu}\bigg|_{\Gamma}=\mathcal{T}u+g.
\end{equation*}
where $\mathcal{T}$ is the boundary operator and
\begin{equation*}
g:=\frac{\partial (u^i+u^\rho)}{\partial \nu}\bigg|_{x_2=0}=2\frac{\partial u^i}{\partial \nu}\bigg|_{x_2=0}\quad \text{on} \ \mathbb{R}\times\Gamma.
\end{equation*}
\end{Definition}

Then we get a new time domain scattering problem: Find $u$ such that
\begin{align}
\label{318}u_{tt}-\Delta u=0  &\quad \text{in}\;\;\mathbb{R}\times D,\\
\label{319}u=0 &\quad \text{on}\;\;\mathbb{R}\times \Lambda,\\
\label{320}\frac{\partial u}{\partial \nu}=\mathcal{T}u+g &\quad \text{on}\;\;\mathbb{R}\times \Gamma,\\
\label{321}u(0,\cdot)=u_t(0,\cdot)=0  &\quad \text{in}\;\; D.
\end{align}


\section{Well-posedness}

In this section, we concern about well-posedness of the scattering problem \eqref{318}--\eqref{321} and the equivalence between the scattering problems \eqref{301}--\eqref{303} and \eqref{318}--\eqref{321}. Instead of the classic equations, the corresponding generalized equations are considered.

For the problem \eqref{301}--\eqref{303}, the set of solutions we care about are $u\in \mathcal{L}'_\sigma(\mathbb{R}_{+},H_{\Delta,E}^1(D\cup\mathbb{R}^2_+))$ such that
\begin{equation}
\label{353}\ddot{u}-\Delta u=f  \quad \text{in}\;\;\mathbb{R}\times (D\cup\mathbb{R}_+^2),
\end{equation}
where $\ddot{u}$ is the generalized second order derivative of $u$ with respect to $t$.

Similarly, what we care about the problem \eqref{318}--\eqref{321} is: find $u\in \mathcal{L}'_\sigma(\mathbb{R}_{+},H_{\Delta,E}^1(D))$ such that
\begin{align}
\label{354}\ddot{u}-\triangle u=0  &\quad \text{in}\;\;\mathbb{R}\times D,\\
\label{355}\frac{\partial u}{\partial \nu}=\mathcal{T}u+g &\quad \text{on}\;\;\mathbb{R}\times \Gamma.
\end{align}

Let $u\in \mathcal{L}'_\sigma(\mathbb{R}_{+},H_{\Delta,E}^1(D))$ be the solution of the problem \eqref{354}--\eqref{355}. The Fourier-Laplace transform implies that $U(\omega,\cdot)\in H_{\Delta,E}^1(D)$ satisfies
\begin{align}
\label{351}\triangle U(\omega,\cdot)+\omega^2 U(\omega,\cdot)=0  &\quad \text{in}\;\;D,\\
\label{352}\frac{\partial U(\omega,\cdot)}{\partial \nu}=\mathcal{T}^F U(\omega,\cdot)+G(\omega,\cdot) &\quad \text{on}\;\; \Gamma.
\end{align}

We first give the following property of the operator $\mathcal{T}^F$.

\begin{Lemma}\label{prop311}
Let $\omega\in\mathbb{C}$. The operator $\mathcal{T}^F :\tilde{H}^{1/2}(\Gamma)\rightarrow H^{-1/2}(\Gamma)$ satisfies
\begin{equation*}
\mathrm{Im}\left(\overline{\omega}\left\langle\mathcal{T}^F V(\omega,\cdot),\overline{V(\omega,\cdot) }\right\rangle_{\Gamma}\right) \geq 0.
\end{equation*}
\end{Lemma}

\begin{proof}
For $V(\omega,\cdot)\in C_0^{\infty}(\Gamma)$, we have
\begin{align*}
&\left\langle\mathcal{T}^F V(\omega,\cdot),\overline{V(\omega,\cdot)}\right\rangle_{\Gamma}\\
=&\frac{\mathrm{i}\omega}{2\pi} \int_{\mathbb{R}}\sqrt{1-\xi_1^2/\omega^2}\mathcal{F}_{x_1}[E_0(V(\omega,\cdot))] \int_{\Gamma}\mathrm{e}^{\mathrm{i}\xi_1x_1}\overline{V(\omega,\cdot)}\mathrm{d}s_x\mathrm{d}\xi_1\\
=&\frac{\mathrm{i}\omega}{2\pi}\int_{\mathbb{R}}\sqrt{1-\xi_1^2/\omega^2}\mathcal{F}_{x_1}[E_0(V(\omega,\cdot) )] \overline{\int_{\mathbb{R}}\mathrm{e}^{-\mathrm{i}\xi_1x_1}E_0(V(\omega,\cdot) )\mathrm{d}x_1}\mathrm{d}\xi_1\\
=&\frac{\mathrm{i}\omega}{2\pi}\int_{\mathbb{R}}\sqrt{1-\xi_1^2/\omega^2}\left|\mathcal{F}_{x_1}[E_0(V(\omega,\cdot) )]\right|^2 \mathrm{d}\xi_1.
\end{align*}
Note that $V(\omega,\cdot)\in C_0^{\infty}(\Gamma)$ is dense in $\tilde{H}^{1/2}(\Gamma)$. Therefore, for $V(\omega,\cdot)\in\tilde{H}^{1/2}(\Gamma)$, we also have
$$\left\langle\mathcal{T}^F V(\omega,\cdot),\overline{V(\omega,\cdot)}\right\rangle_{\Gamma}= \frac{\mathrm{i}\omega}{2\pi}\int_{\mathbb{R}}\sqrt{1-\xi_1^2/\omega^2} \left|\mathcal{F}_{x_1}[E_0(V(\omega,\cdot))]\right|^2 \mathrm{d}\xi_1.$$
Thus
\begin{align*}
\mathrm{Im}\left(\overline{\omega}\left\langle\mathcal{T}^F V(\omega,\cdot),\overline{V(\omega,\cdot) }\right\rangle_{\Gamma}\right) =&\mathrm{Re}\left(\frac{|\omega|^2}{2\pi}\int_{\mathbb{R}}\sqrt{1-\xi_1^2/\omega^2}| \mathcal{F}_{x_1}[E_0(V(\omega,\cdot))]|^2 \mathrm{d}\xi_1\right)\\
\geq &0.
\end{align*}
This completes the proof.
\end{proof}

Moreover, we need the following property of the norms given by \cite{sayas2011retarded}.
Let $\omega\in\mathbb{C}$ and $\mathrm{Im}(\omega)=\sigma>0$. Then the norms $\normmm{\cdot}_{|\omega|,D}$ and $\|\cdot\|_{H^1(D)}$ satisfy
\begin{equation}\label{360}
\underline{\sigma}\|V\|_{H^1(D)}\leq\normmm{V}_{|\omega|,D}\leq \frac{|\omega|}{\underline{\sigma}}\|V\|_{H^1(D)}.
\end{equation}

Then we give the following theorem about the well-posedness of the problem \eqref{351}--\eqref{352}.

\begin{Proposition}\label{prop313}
Let $\omega\in\mathbb{C}$, $\mathrm{Im}(\omega)=\sigma>0$ and $G(\omega,\cdot)\in H^{-1/2}(\Gamma)$. There exists a unique solution $U(\omega,\cdot)\in H_{\Delta,E}^1(D)$ of the problem \eqref{351}--\eqref{352}. Moreover, there exists a constant $C_{\sigma,D}$ depending only on $\sigma$ and $D$ such that
$$\|U(\omega,\cdot)\|_{H^{1}_{\Delta,E}(D)}\leq C_{\sigma,D}|\omega|^2\|G(\omega,\cdot)\|_{H^{-1/2}(\Gamma)}.$$
\end{Proposition}

\begin{proof}
Consider the variational formulation associate with the problem \eqref{351}--\eqref{352}: $U(\omega,\cdot)\in H_E^1(D)$ solves \eqref{351}--\eqref{352} if and only if
\begin{align*}
A\left(U(\omega,\cdot),\overline{V(\omega,\cdot)}\right)
:=&\left(\nabla U(\omega,\cdot),\nabla \overline{V(\omega,\cdot)}\right)_{D}-\omega^2\left(U(\omega,\cdot),\overline{V(\omega,\cdot)}\right)_{D} -\left\langle\mathcal{T}^F(\omega)U(\omega,\cdot), \overline{V(\omega,\cdot)}\right\rangle_{\Gamma}\\
=&\left\langle G(\omega,\cdot),\overline{V(\omega,\cdot)}\right\rangle_{\Gamma},\quad \forall V(\omega,\cdot)\in  H_E^1(D).
\end{align*}
It follows from Lemma \ref{prop311} and \eqref{360} that
\begin{align*}
\mathrm{Im}\left(-\overline{\omega} A\left(U(\omega,\cdot),\overline{U(\omega,\cdot)}\right)\right)
=& \mathrm{Im}\Big(-\overline{\omega}\left(\nabla U(\omega,\cdot),\nabla \overline{U(\omega,\cdot)}\right)_{D} +\overline{\omega}\omega^2\left(U(\omega,\cdot),\overline{U(\omega,\cdot)}\right)_{D} \\
&+\overline{\omega}\left\langle\mathcal{T}^F(\omega)U(\omega,\cdot),\overline{U(\omega,\cdot)}\right\rangle_{\Gamma}\Big)\\
\geq &\sigma\|\nabla U(\omega,\cdot)\|_{D}^2+\sigma|\omega|^2\|U(\omega,\cdot)\|_{D}^2\\
= & \sigma\normmm{U(\omega,\cdot)}_{|\omega|,D}^2.
\end{align*}
Then the unique solvability of the problem \eqref{351}--\eqref{352} in $H^1_E(D)$ follows from a Lax-Milgram argument. Moreover, it follows from the trace theorem and \eqref{360} that
\begin{align*}
\normmm{U(\omega,\cdot)}_{|\omega|,D}^2\leq &\frac{1}{\sigma}\mathrm{Im}\left(-\overline{\omega} \left\langle G(\omega,\cdot),\overline{U(\omega,\cdot)}\right\rangle_{\Gamma}\right) \\
\leq & \frac{|\omega|}{\sigma}\|G(\omega,\cdot)\|_{H^{-1/2}(\Gamma)} \|U(\omega,\cdot)\|_{\tilde{H}^{1/2}(\Gamma)}\\
\leq & C_{D}\frac{|\omega|}{\sigma}\|G(\omega,\cdot)\|_{H^{-1/2}(\Gamma)} \|U(\omega,\cdot)\|_{H_E^1(D)}\\
\leq & C_{D}\frac{|\omega|}{\sigma\underline{\sigma}}\|G(\omega,\cdot)\|_{H^{-1/2}(\Gamma)} \normmm{U(\omega,\cdot)}_{|\omega|,D},
\end{align*}
where $C_D$ is a constant depending only on the domain $D$. Thus
$$\normmm{U(\omega,\cdot)}_{|\omega|,D}^2\leq C_{D}\frac{|\omega|}{\sigma\underline{\sigma}}\|G(\omega,\cdot)\|_{H^{-1/2}(\Gamma)}.$$
Then \eqref{360} implies
$$\|U(\omega,\cdot)\|_{H_E^1(D)}^2\leq C_{D}\frac{|\omega|}{\sigma\underline{\sigma}^2}\|G(\omega,\cdot)\|_{H^{-1/2}(\Gamma)}.$$

We have proved that there exists a unique solution $U(\omega,\cdot)\in H_{E}^1(D)$ of the problem \eqref{351}--\eqref{352}. Then \eqref{351} implies that $\Delta U(\omega,\cdot)=-\omega^2 U(\omega,\cdot)\in H^1_E(D)\subset L^2(D)$. Thus $U(\omega,\cdot)\in H_{\Delta,E}^1(D)$, which means there exists a unique solution $U(\omega,\cdot)\in H_{\Delta,E}^1(D)$ of the problem \eqref{351}--\eqref{352}. Moreover, the definition of the norm $\normmm{\cdot}_{c,\omega}$ implies
\begin{align*}
\|\Delta U(\omega,\cdot)\|_D=&\|\omega^2 U(\omega,\cdot)\|_D \leq |\omega|^2\|U(\omega,\cdot)\|_D\\
\leq& |\omega|\normmm{U(\omega,\cdot)}_{|\omega|,D} \leq C_D \frac{|\omega|^2}{\sigma\underline{\sigma}}\|G(\omega,\cdot)\|_{H^{-1/2}(\Gamma)}.
\end{align*}
Then we have
\begin{align*}
\|U(\omega,\cdot)\|_{H_{\Delta,E}^1(D)}^2=&\|U(\omega,\cdot)\|_{H_{E}^1(D)}^2+\|\Delta U(\omega,\cdot)\|_{D}^2 \\
\leq& C_D^2 \frac{|\omega|^2}{\sigma^2\underline{\sigma}^2} \left(\frac{1}{\underline{\sigma}^2}+|\omega|^2\right) \|G(\omega,\cdot)\|_{H^{-1/2}(\Gamma)}^2\\
\leq& 2C_D^2 \frac{|\omega|^4}{\sigma^2\underline{\sigma}^6} \|G(\omega,\cdot)\|_{H^{-1/2}(\Gamma)}^2.
\end{align*}
Thus
$$\|U(\omega,\cdot)\|_{H^{1}_{\Delta,E}(D)}\leq C_{\sigma,D}|\omega|^2\|G(\omega,\cdot)\|_{H^{-1/2}(\Gamma)},$$
where $C_{\sigma,D}=\sqrt{2} C_D/(\sigma\underline{\sigma}^3)$. This completes the proof.
\end{proof}

Then we consider the time domain scattering problem \eqref{354}--\eqref{355}. We need the following lemma for the time domain analysis.

\begin{Lemma}~\textup{(\cite{lubich1994multistep})}\label{prop315}
Let $p,r\in\mathbb{R}$, $\omega\in\mathbb{C}$ and $\mathrm{Im}(\omega)>\sigma_0>0$. $F(\omega)\in\mathcal{B}(X,Y)$ is a bounded operator between the Hilbert spaces  $X$ and $Y$. Define $f(t):=\mathcal{L}^{-1}[F(\omega)](t)$ and $F_t g:=\int_{-\infty}^{\infty}f(t)g(\cdot-t)\mathrm{d}t$. Assume that
$$\|F(\omega)\|_{\mathcal{B}(X,Y)}\leq C|\omega|^r,\quad \mathrm{Im}(\omega)>\sigma_0.$$
Then, for $\sigma>\sigma_0$, $F_t$ is a bounded operator from $H_{\sigma}^{p+r}(\mathbb{R}_+,X)$ to $H_{\sigma}^{p}(\mathbb{R}_+,Y)$.
\end{Lemma}

We have the following results in time domain.

\begin{Theorem}\label{prop314}
Let $\sigma>\sigma_0>0$, $p\in\mathbb{R}$ and $g\in H_{\sigma}^{p+2}(\mathbb{R}_+, H^{-1/2}(\Gamma))$. Then there exists a unique solution $u\in H_{\sigma}^{p}(\mathbb{R}_+,H_{\Delta,E}^{1}(D))$ of the scattering problem \eqref{354}--\eqref{355}. Moreover, there exists a constant $C'_{\sigma_0,D}$ depending only on $\sigma_0$ and $D$ such that
$$\|u\|_{H_{\sigma}^{p}(\mathbb{R}_+,H_{\Delta,E}^{1}(D))}\leq C'_{\sigma_0,D}\|g\|_{H_{\sigma_0}^{p+2}(\mathbb{R}_+, H^{-1/2}(\Gamma))}.$$
\end{Theorem}

\begin{proof}
For $\mathrm{Im}(\omega)>\sigma_0>0$, it follows from Theorem \ref{prop313} that there exists a unique solution $U(\omega,\cdot)\in H_{\Delta,E}^1(D)$ of the problem \eqref{351}--\eqref{352} and
$$\|U(\omega,\cdot)\|_{H^{1}_{\Delta,E}(D)}\leq C_{\sigma_0,D}|\omega|^2\|G(\omega,\cdot)\|_{ H^{-1/2}(\Gamma)},$$
where $C_{\sigma_0,D}=2C_D/(\sigma_0\underline{\sigma_0}^3)$, here $C_D$ is the same constants in the proof of Proposition \ref{prop313}. Denote by $F(\omega,\cdot)\in\mathcal{B}( H^{-1/2}(\Gamma),H^{1}_E(D))$ the solution operator of the problem \eqref{351}--\eqref{352} such that $U(\omega,\cdot)=F(\omega,\cdot)G(\omega,\cdot)$. Then
$$\|F(\omega,\cdot)\|_{\mathcal{B}( H^{-1/2}(\Gamma),H^{1}_{\Delta,E}(D))}\leq C_{\sigma_0,D}|\omega|^2,\quad \mathrm{Im}(\omega,\cdot)>\sigma_0.$$
Using Theorem \ref{prop313} and Lemma \ref{prop315}, an inverse Fourier-Laplace argument implies that $u=f*g\in H_{\sigma}^{p}(\mathbb{R}_+,H_E^{1}(D))$ is the unique solution of the scattering problem \eqref{354}--\eqref{355}, in which $f$ and $g$ are the inverse Fourier-Laplace transform of $F(\omega,\cdot)$ and $G(\omega,\cdot)$, respectively. Moreover, Lemma \ref{prop315} implies
$$\|u\|_{H_{\sigma}^{p}(\mathbb{R}_+,H_{\Delta,E}^{1}(D))}\leq C'_{\sigma_0,D}\|g\|_{H_{\sigma_0}^{p+2}(\mathbb{R}_+, H^{-1/2}(\Gamma))}.$$
where $C'_{\sigma_0,D}$ is a constant depending only on $\sigma_0$ and $D$.
\end{proof}

At the end of this section, we provide the following proposition of the equivalence between the time domain scattering problems \eqref{353} and \eqref{354}--\eqref{355}.
\begin{Proposition}\label{prop301}
Let $\sigma>0$, $p\in\mathbb{R}$ and $g\in H_{\sigma}^{p+2}(\mathbb{R}_+, H^{-1/2}(\Gamma))$. If $u_1\in H_{\sigma}^{p}(\mathbb{R}_+,H^{1}_{\Delta,E}(D\cup\mathbb{R}^2_+))$ is the solution of the scattering problem \eqref{353} and $u_2\in H_{\sigma}^{p}(\mathbb{R}_+,H_{\Delta,E}^{1}(D))$ is the solution of the scattering problem \eqref{354}--\eqref{355}. Then
$$u_1=u_2 \quad \textup{in}\;\;\mathbb{R}\times D.$$
\end{Proposition}

\begin{proof}
If $u_1\in H_{\sigma}^{p}(\mathbb{R}_+,H^{1}_{\Delta,E}(D\cup\mathbb{R}^2_+))$ is the solution of the problem  \eqref{353}, combining $\mathrm{supp}(f)\subset\mathbb{R}_+\times\mathbb{R}_+^2$ with the analysis of the TBC \eqref{355}, we can get that $u_1|_{\mathbb{R}\times D}\in H_{\sigma}^{p}(\mathbb{R}_+,H_{\Delta,E}^{1}(D))$ is the solution of the problem \eqref{354}--\eqref{355}.

Using Theorem \ref{prop314}, the unique solvability of the scattering problem \eqref{354}--\eqref{355} implies
$$u_1=u_2 \quad \text{in}\;\;\mathbb{R}\times D.$$
The proof is completed.
\end{proof}


\section{Boundary integral equations}

In this section, we will show the procedure to solve the time domain scattering problem \eqref{318}--\eqref{321} and the CQ method to turn the calculation of the time domain problem into that of the corresponding frequency domain problems.

The retarded single layer potential on $\partial D$ is defined as (see \cite{Chen2010sampling})
\begin{equation*}
(SL_{\partial D}\phi)(t,x):=\int_{\partial D}\int_0^{t}k(t-\tau,|x-y|)\phi(y,\tau)\,\mathrm{d}\tau \mathrm{d}s_y,\quad t\in \mathbb{R},\, x\in \mathbb{R}^2\backslash\partial D.
\end{equation*}
The retarded double layer potential on $\partial D$ is
\begin{equation*}
(DL_{\partial D}\phi)(t,x):=\int_{\partial D}\int_0^{t}\frac{\partial k(t-\tau,|x-y|)}{\partial\nu(y)} \phi(y,\tau)\,\mathrm{d}\tau \mathrm{d}s_y,\quad t\in \mathbb{R},\, x\in \mathbb{R}^2\backslash\partial D,
\end{equation*}
where $\partial k/\partial {\nu(y)}$ is the normal derivative of $k$ on $\partial D$ with respect to $y$. Also of importance are the single and double layer operators on $\partial D$ defined as
\begin{equation*}
(S_{\partial D}\phi)(t,x):=\int_{\partial D}\int_0^{t}k(t-\tau,|x-y|) \phi(y,\tau)\,\mathrm{d}\tau \mathrm{d}s_y,\quad t\in \mathbb{R},\, x\in \partial D
\end{equation*}
and
\begin{equation*}
(K_{\partial D}\phi)(t,x):=\int_{\partial D}\int_0^{t}\frac{\partial k(t-\tau,|x-y|)}{\partial\nu(y)} \phi(y,\tau)\,\mathrm{d}\tau \mathrm{d}s_y,\quad t\in \mathbb{R},\, x\in \partial D,
\end{equation*}
respectively.

Denote by $\gamma^- u$ and $\gamma^+ u$ the restriction of $u$ to $\partial D$ from interior and exterior and by $\partial_{\nu}^- u$ and $\partial_{\nu}^+ u$ the normal derivatives on $\partial D$ from interior and exterior, respectively. Then the jumps are defined as
\begin{equation*}
[\![u]\!]:=\gamma^- u-\gamma^+ u,\quad [\![\partial_{\nu} u]\!]:=\partial_{\nu}^- u-\partial_{\nu}^+ u.
\end{equation*}

The Kirchhoff's formula (\cite{sayas2011retarded}) for the solution of the wave equation is
\begin{equation}\label{322}
u=SL_{\partial D}[\![\partial_{\nu} u]\!]-DL_{\partial D}[\![u]\!] \quad \text{in}\;\;\mathbb{R}\times D.
\end{equation}
On the boundary $\partial D$, we have
$$\frac{1}{2}u=S_{\partial D}[\![\partial_{\nu} u]\!]-K_{\partial D}[\![u]\!]\quad \text{on}\;\;\mathbb{R}\times \partial D.$$
We are concerned with the time domain scattering problem in the bounded domain $D$. Assume that $u\equiv 0$ in $\mathbb{R}\times (\mathbb{R}^2\backslash\overline{D})$, thus
$$[\![\partial_{\nu} u]\!]=\partial_{\nu}^- u,\quad [\![u]\!]=\gamma^- u\quad \text{on}\;\;\mathbb{R}\times \partial D.$$
For the sake of simplicity, we write $\partial_{\nu} u=\partial_{\nu}^- u$, $u=\gamma^- u$ on $\partial D.$ Then we have the following RPBIEs for the scattering problem \eqref{318}--\eqref{321}:
\begin{align}
\label{323}\frac{1}{2}u&=S_{\Gamma}(\mathcal{T}u+g)-K_{\Gamma}u+S_\Lambda\partial_{\nu} u\quad \text{on} \; \mathbb{R}\times\Gamma,\\
\label{324}0&=S_{\Gamma}(\mathcal{T}u+g)-K_{\Gamma}u+S_\Lambda\partial_{\nu} u\quad \text{on} \; \mathbb{R}\times \Lambda.
\end{align}

We can get $u|_{\mathbb{R}\times\Gamma}$ and $\partial_{\nu} u|_{\mathbb{R}\times \Lambda}$ by solving \eqref{323}--\eqref{324}. Then $\partial_{\nu} u|_{\mathbb{R}\times\Gamma}$ is given by \eqref{320} and $u|_{\mathbb{R}\times D}$ is given by the Kirchhoff's formula \eqref{322}.

We recall the CQ method (\cite{banjai2008rapid,lubich1988convolution}) for the time discretization of the RPBIEs \eqref{323}--\eqref{324}. The time discretization is implemented in $t\in[0,T]$. The terminal time $T$ is chosen such that the energy of the scattered data inside the interested domain is negligible when $t>T$. We have the discretization
$$t_j=j\kappa,\ j=0,1,\ldots,N,\ \kappa=T/N.$$

To solve an integral equation with convolution structure such as
$$S_{\Gamma}u=h,\quad \text{on}\;\;\mathbb{R}\times\Gamma$$
The CQ method leads to the decoupled problems (\cite{banjai2008rapid})
\begin{equation*}
S^F_{\Gamma}\hat{u}_l(\omega_l,x)=\hat{h}_l(\omega_l,x),\quad x\in\Gamma,\,l=0,1,\ldots,N,
\end{equation*}
where $S^F_{\Gamma}$ is denoted as the Fourier-Laplace transform of the operator $S_{\Gamma}$, $\hat{u}_l(\cdot)$ and $\hat{h_l}$ are, respectively, the discrete Fourier transform of $u_j(\cdot):=u(t_j,\cdot)$ and $h_j:=h(t_j)$ with respect to $j$, $\omega_l\in\mathbb{C}$ are constants depending on the time discretization. We choose
$$\omega_l=\frac{\mathrm{i}}{2\kappa}(\xi_l^2-4\xi_l+3),$$
where
$$\xi_l=\gamma\mathrm{e}^{-\mathrm{i}\frac{2\pi l}{N_T}},\quad l=0,\ldots,N_T-1.$$
In this paper, we suggest to use the same strategy as that in \cite{banjai2008rapid} to choose the stability parameter $\gamma$.

Notice that
\begin{align*}
\mathcal{T}v&=\frac{1}{2\pi}\int_\mathbb{R}\mathrm{e}^{\mathrm{i}\xi_1x_1}\mathcal{L}^{-1}\left[ \mathrm{i}\omega\sqrt{1-\xi_1^2/\omega^2}\mathcal{L}\circ\mathcal{F}_{x_1}[E_0(v)](\omega,\xi_1,0)\right]\mathrm{d}\xi_1\\
&=\frac{1}{2\pi}\int_\mathbb{R}\mathrm{e}^{\mathrm{i}\xi_1x_1}\mathcal{L}^{-1}\left[ \mathrm{i}\omega\sqrt{1-\xi_1^2/\omega^2}\right]*\mathcal{F}_{x_1}[E_0(v)](t,\xi_1,0)\mathrm{d}\xi_1,
\end{align*}
and
\begin{equation*}
g=2\frac{\partial u^i}{\partial \nu}=2\partial_{\nu} k(t,|x-z|)*\lambda(t),\quad x\in\mathbb{R}^2_0,
\end{equation*}
where $z:=(z_{1},z_{2})$ is the source point. Then we get
\begin{equation*}
S_{\Gamma}\mathcal{T}v=\frac{1}{2\pi}\int_\Gamma\int_\mathbb{R}\mathrm{e}^{\mathrm{i}\xi_1x_1} k(t,|x-y|)*\mathcal{L}^{-1}\left[ \mathrm{i}\omega\sqrt{1-\xi_1^2/\omega^2}\right]*\mathcal{F}_{x_1}[E_0(v)](t,\xi_1,0)\mathrm{d}\xi_1\mathrm{d}s_y
\end{equation*}
and
\begin{equation*}
S_{\Gamma}g=2\int_\Gamma k(t,|x-y|)*\partial_{\nu} k(t,|y-z|)*\lambda(t)\mathrm{d}s_y.
\end{equation*}

To solve the integral equations \eqref{323}--\eqref{324}, the convolution quadrature method leads to the decoupled problems
\begin{align}
\label{328}\frac{1}{2}\hat{u}_l(\omega_l,\cdot)=S^F_{\Gamma}\{\mathcal{T}^F\hat{u}_l+G^F\hat{\lambda}_l\} (\omega_l,\cdot) -K^F_{\Gamma}\hat{u}_l(\omega_l,\cdot)+S^F_\Lambda\partial_{\nu} \hat{u}_l(\omega_l,\cdot)&\quad \text{on}\;\; \Gamma,\\
\label{329}0=S^F_{\Gamma}\{\mathcal{T}^F\hat{u}_l+G^F\hat{\lambda}_l\} (\omega_l,\cdot) -K^F_{\Gamma}\hat{u}_l(\omega_l,\cdot)+S^F_\Lambda\partial_{\nu} \hat{u}_l(\omega_l,\cdot)&\quad \text{on}\;\; \Lambda,
\end{align}
where $l=0,1,\ldots,N$, $\hat{u}_l(\cdot)$ and $\hat{\lambda}_l(\cdot)$ are, respectively, the discrete Fourier transform of $u_j(\cdot):=u(t_j,\cdot)$ and $\lambda_j:=\lambda(t_j)$ with respect to $j$. The operators are
\begin{align*}
(S^F_{\Gamma}\varphi)(\omega,x)&=\frac{\mathrm{i}}{4}\int_{\Gamma}H_0^{(1)}(\omega|x-y|)\varphi(\omega,y)\mathrm{d}s_y,\\
(K^F_{\Gamma}\varphi)(\omega,x)&=\frac{\mathrm{i}}{4}\int_{\Gamma}\partial_{\nu(y)}H_0^{(1)}(\omega|x-y|)\varphi(\omega,y)\mathrm{d}s_y,\\
(\mathcal{T}^F\varphi)(\omega,x)&=\frac{\mathrm{i}\omega}{2\pi}\int_\mathbb{R}\mathrm{e}^{\mathrm{i}\xi_1x_1} \sqrt{1-\xi_1^2/\omega^2}\mathcal{F}_{x_1}(\varphi)(\omega,\xi_1,x_2)\mathrm{d}\xi_1,\\
(G^F\varphi)(\omega,x)&=-\frac{\mathrm{i}\omega}{2}H_1^{(1)}(\omega|x-z|)\frac{x_2-z_{2}}{|x-z|}\varphi(\omega),
\end{align*}
where $H_n^{(1)}$ is the Hankel function of the first kind of order $n$.

Finally, we just need to solve the Helmholtz problems \eqref{328}--\eqref{329} instead of the time domain scattering problem \eqref{323}--\eqref{324}.


\section{Conclusion}

We have analyzed the time domain acoustic scattering from open cavities. A TBC have been developed to get an equivalent initial boundary value problem. The well-posedness of the reduced scattering problem have been proved. Moreover, RPBIEs have been established to solve the reduced problem.

Our future work will include the analysis of the existence and uniqueness of the solutions for the RPBIEs and the iteration method for the inverse scattering problem.


\section*{Acknowledgements}

The work of Bo Chen was supported by the NSFC [No. 11671170] and the Scientific Research Foundation of Civil Aviation University of China [No. 2017QD04S]. The work of Fuming Ma was supported by the NSFC [No. 11771180]. The work of Yukun Guo was supported by the NSFC [No. 11601107, 41474102 and 11671111].

\bibliographystyle{abbrv} 
\bibliography{reference}

\end{document}